\documentclass[12pt]{article}

\usepackage{pdfpages}
\usepackage{amsmath}
\usepackage{amssymb}
\usepackage{amsmath}
\usepackage{amsfonts}
\usepackage{upgreek}
\usepackage{sgame}
\usepackage{accents}
\usepackage{amsthm}
\usepackage{amssymb}
\usepackage{tikz-cd}
\usepackage{authblk}
\usepackage{epigraph}

\newtheorem{theorem}{Theorem}
\newtheorem{corollary}{Corollary}
\newtheorem{lemma}{Lemma}
\newtheorem{proposition}{Proposition}

\newtheorem{example}{Example}

\newtheorem{remark}{Remark}

\usepackage{makeidx}

\usepackage{hyperref}
\hypersetup{
    colorlinks=true,
    linkcolor=blue,
    filecolor=magenta,      
    urlcolor=cyan,
    citecolor=red
}

\makeatother

\begin{document}

\author[1]{Artem Hulko}
\author[2]{Mark Whitmeyer\thanks{Email: \href{mailto:mark.whitmeyer@utexas.edu}{mark.whitmeyer@utexas.edu}}}
\affil[1]{Department of Mathematics and Statistics, University of North Carolina at Charlotte}
\affil[2]{Department of Economics, University of Texas at Austin}

\title{A Game of Nontransitive Dice}

\date{\today{}}

\maketitle

\begin{abstract}

We consider a two player simultaneous-move game where the two players each select any permissible $n$-sided die for a fixed integer $n$. A player wins if the outcome of his roll is greater than that of his opponent. Remarkably, for $n>3$, there is a unique Nash Equilibrium in pure strategies. The unique Nash Equilibrium is for each player to throw the Standard $n$-sided die, where each side has a different number. Our proof of uniqueness is constructive. We introduce an algorithm with which, for any nonstandard die, we may generate another die that beats it.

\end{abstract}

\newpage
\section{Introduction}\setlength{\epigraphwidth}{3in} 
\epigraph{If Hercules and Lychas play at dice\\ Which is the better man, the greater throw\\ May turn by fortune from the weaker hand.}{William Shakespeare,\\ The Merchant of Venice}

Nontransitive dice are a fascinating topic in applied probability. They first came into the limelight as a result of a paper by Martin Gardner \cite{g970} and are one of a larger class of nontransitivity ``paradoxes" (see \cite{t976}, \cite{b972}), which also include the well-known Condorcet Voting Paradox, as described in \cite{f983}.

Recent papers published concerning nontransitive dice include \cite{a016}, \cite{c016}, and \cite{s994}. Another recent interpretation of the situation is \cite{h016}, which instead reinterprets the scenario through throws of unfair \textit{coins}. Additionally, the issue of nontransitive dice is the subject of a recent polymath project paper (see \cite{g017}), and indeed we borrow some terminology from that paper.

One problem that may be considered is how nontransitive dice should be played in a strategic interaction of two or more players. \cite{r001} was the first to explore this: it explores the two player game where each player could choose one of four specific\footnote{\cite{r001} formulates the problem using the four six-sided ``Efron Dice".} non-transitive symmetric dice and finds the set of equilibria. It then extends the analysis to cover the situation in which the players each choose two dice.

Here, we investigate a broader problem:  we consider a two-player, one-shot, simultaneous move game in which each player selects a general $n$-sided die and rolls it. The player with the highest face showing wins a reward, which we may normalize to $1$. The solution concept that we use for the game is that of a (pure strategy) Nash Equilibrium.\footnote{Note that this game is a constant-sum game; therefore equivalent to a zero sum game, and for two players the concept of a Nash Equilibrium is equivalent to that of a saddle point.} We show that for $n > 3$ there is a single, unique\footnote{There may be additional mixed strategy equilibria; however, in this analysis we focus exclusively on pure strategy equilibria. Henceforth, by Nash Equilibrium or equilibrium, we mean only those in pure strategies.} Nash Equilibrium in which both players play the \hyperlink{standard}{standard} die. Moreover, our proof of uniqueness is constructive and contains an algorithm that, for any non-standard die, generates a die that beats it. One notable implication of this is that for any non-standard die, there is at least one die that is the result of a \hyperlink{OS}{one-step}, the most elementary step of our algorithm, applied to the standard die, that beats it. That is, for any die that is not the standard die, there is a die very similar to the standard die that beats it.

The closest paper to this one, \cite{f007}, considers the same problem, where for some fixed integer $n$, two players each choose a die and roll against each other. It also shows that the \hyperlink{standard}{standard} die ties any other die in expectation, and that every nonstandard die loses to some other die. \cite{de} explores dice games as well in a slightly more general setting, and the existence and uniqueness of the the equilibrium in which both players each roll the standard die follows from Propositions 6 and 8 in that paper.

Our paper differs from \cite{f007} and \cite{de} in the following key ways. We provide different proofs of the existence and uniqueness of the Nash Equilibrium in the game\footnote{\cite{f007} do not explicitly show this, but it clearly follows from their Lemma 1 and Theorem 2.}, and we are able to do so using exclusively elementary mathematics. Additionally, our proof is constructive and we formulate a simple algorithm that allows us, for any non-standard die, to generate a die that beats it. Moreover, our last result--that for any non standard die, there is a die that beats it that is merely a one-step away from the standard die--is also novel.

Finally, dice games can be placed in a more general context, as a member of the family of Colonel Blotto games. First developed by E. Borel in 1921 (see \cite{borel}), a burgeoning literature has resulted, due to the game's general applications in economics, operations research, political science, and other areas. Some recent papers include \cite{rob} and \cite{hart}. In addition, recently, \cite{whit} explores a $n$ player continuous version of this game played on the interval $[0,1]$. For two players, the unique equilibrium is the continuous analog of the unique equilibrium here, the uniform distribution.

\section{The Basic Game}\label{tbg}

Define a general $n$-sided die (henceforth just ``die") as an integer-valued random variable $D_{n}$ that takes values in the finite set $\big\{1,2,...,n\big\}$, where the distribution must satisfy the following conditions:

\begin{enumerate}
\item For each possible value of $D_{n}$, $d_{i}$, the probability that it occurs, $p_{d_{i}}$, is a multiple of $1/n$.
\item \[\mathbb{E}_{p}(D_{n}) = \sum_{i=1}^{n}d_{i} p_{d_{i}} = \frac{n+1}{2}\]
\end{enumerate}
For a given $n$, denote the set of all $n$-sided dice by $\mathcal{D}_{n}$.

Then, a \hypertarget{standard}{standard} $n$ -sided die, $S_{n}$, is simply a die where each value occurs with equal probability, $1/n$. Naturally $S_{n} \in \mathcal{D}_{n}$.

We can represent any $n$-sided die, $D_{n} \in \mathcal{D}_{n}$, as a (discrete) uniformly distributed random-variable that takes values in the multiset of size $n$, with elements in $\big\{1,2,...,n\big\}$ and sum equal to $\frac{n(n+1)}{2}$:

\begin{equation}\label{eq:2}
\begin{split}
    D_{n} &= \big\{[d_{1}, d_{2},\dots,d_{n}]\big\}
\end{split}
\end{equation}

\begin{example}
\label{ex1}
The five $4$-sided dice are $S_{4} = \big\{[1,2,3,4]\big\}$, $X_{4} = \big\{[1,1,4,4]\big\}$, $Y_{4} = \big\{[2,2,2,4]\big\}$, $\big\{[1,3,3,3\big\}$, $\big\{[2,2,3,3]\big\}$.
\end{example}


\subsection{The Game}

Consider two players, Amy ($A$) and Bob ($B$). They play the following one shot game. Fix $n$, Amy and Bob each independently select any $n$-sided die, $A_{n}, B_{n} \in \mathcal{D}_{n}$ and then roll them against each other. Their expected payoffs are the probability that the realization of their roll is higher than the realization of their opponents roll\footnote{If the realizations of the rolls are the same, the winner is decided by a (fair) coin flip.}.

A \textbf{Strategy} for Amy, (and analogously for Bob), is simply a choice of die $A_{n} \in \mathcal{D}_{n}$. For any pair of strategies, $(A_{n},B_{n})$, Amy's expected payoff, $U_{A}(A_{n},B_{n})$, is 

\[
U_{A}(A_{n},B_{n}) = \Pr(A_{n}>B_{n}) + \frac{1}{2}\Pr(A_{n}=B_{n})
\]

\begin{example}\label{ex2}
Suppose $n = 4$ and let Amy and Bob choose dice $X_{4}$ and $Y_{4}$ from Example \ref{ex1}, respectively. Then $U_{A}(X_{4},Y_{4}) = 7/16$ and $U_{B}(X_{4},Y_{4}) = 9/16$.
\end{example}

The main result of this paper is the following theorem.

\begin{theorem}\label{thm1}
For any $n$, the unique Nash Equilibrium of the two player game is where both players play the standard die $S_{n}$. That is, the unique Nash Equilibrium is the strategy pair $(S_{n},S_{n})$.
\end{theorem}

We shall prove this theorem by proving two propositions: Proposition \ref{prop1}, that $(S_{n},S_{n})$ is a Nash Equilibrium; and Proposition \ref{prop2}, that $(S_{n},S_{n})$  is the unique equilibrium (for $n \geq 4$).

\begin{proposition}
\label{prop1} $(S_{n},S_{n})$ is a Nash Equilibrium.
\end{proposition}
\begin{proof}
First, we show that for all $i \in \big\{A,B\big\}$ and for all $D_{n} \in \mathcal{D}_{n}$, $U_{i}(S_{n},D_{n}) = U_{i}(D_{n},S_{n}) = \frac{1}{2}$.

Suppose player $B$ chooses the standard die, $S_{n}$ and player $A$ chooses an arbitrary die $D_{n} = \big\{[d_{1}, d_{2},\dots,d_{n}]\big\}$. If the realization of $D_{n}$ is $d_{i}$; that is, dice $D_{n}$ ``lands" showing face $d_{i}$, then with probability $(d_{i} - 1)/n$, $D_{n}$ beats the standard die, and with probability $1/n$, $D_{n}$ ties the standard die. Hence,

\[U_{A}(D_{n},S_{n}) = \sum_{i=1}^{n}\bigg(\frac{1}{n}\bigg)\bigg(\frac{d_{i}-1}{n} + \frac{1}{n}\cdot \frac{1}{2}\bigg) = \bigg(\frac{1}{n}\bigg)\bigg(\mathbb{E}[D_{n}\big] - \frac{1}{2}\bigg) = \frac{1}{2}\]

Since $(S_{n},S_{n})$ yields to player $B$ a payoff of $1/2$, there is no profitable deviation.
\end{proof}

For any die $D_{n} \in \mathcal{D}_{n}$ and any $i$ and $j$ such that $i \neq j$, $d_{i} \neq 1$ and $d_{j} \neq n$, a \hypertarget{OS}{\textbf{One-step}} from $D_{n}$ is $\Phi_{i,j}(D_{n})$ defined by

\[
\Phi_{i,j}(D_{n}) = \big\{[d_{1},\dots, d_{i-1}, d_{i}-1, d_{i+1},\dots, d_{j-1}, d_{j} + 1, d_{j+1}, \dots, d_{n}]\big\}
\]

Evidently, $\Phi_{i,j}(D_{n})$ satisfies the conditions in Equation \ref{eq:2} and so $\Phi_{i,j}(D_{n})$ is a die. We may make the following remark:

\begin{remark}
\label{rem1} Any die can be reached in a sequence of one-steps from any other die. This of course implies that any die can be reached in a sequence of one-steps from the standard die, $S_{n}$.
\end{remark}

For any two dice $A_{n}, B_{n} \in \mathcal{D}_{n}$, we say $A_{n}$ \textbf{Beats} $B_{n}$ if the number of pairs $(a_{i},b_{j})$ with $a_{i}>b_{j}$ exceeds the number of pairs with $a_{i}<b_{j}$.

It remains to show uniqueness, which we accomplish in the following lemma.

\begin{proposition}
\label{prop2} The Nash Equilibrium, $(S_{n},S_{n})$, is unique for $n \geq 4$.
\end{proposition}

\begin{proof}
Clearly, for any strategy pair, player $i$ has a profitable deviation if and only if there is a die that beats her opponents die. Our proof is constructive and we show that for any die $B_{n} \neq S_{n}$, we can construct a die, ${G}_{n}$, that beats it.

To that end, let $B_{n} = \big\{[b_{1},b_{2},\dots,b_{n}]\big\}$, and (recall) $S_{n} = \big\{[1,2,\dots,n]\big\}$. We introduce $\gamma_{k}$, defined as

\[
\begin{split}
\gamma_{k} &= |\big\{b_{i} |b_{i} = k \big\}|
\end{split}\]
for $k = 1,2,\dots,n$. Note that each $\gamma$ is a non-negative integer and that each satisfies

\[\begin{split}
\sum_{k=1}^{n}\gamma_{k} &= n
\end{split}\]

\[
\begin{split}
\sum_{k=1}^{n}k\gamma_{k} &= \frac{n(n+1)}{2}
\end{split}
\]
Next, for $k = 1, 2,\dots,n-1$ define $\xi_{k}$ as

\[
\xi_{k} = \gamma_{k} + \gamma_{k+1}
\]

To construct a die $G_{n}$ that beats $B_{n}$, we need simply find a pair $(\xi_{i}, \xi_{j})$ with $\xi_{i} > \xi_{j}$ (and so clearly $j \neq i$) and $i \neq j+1$. Then, simply add $1$ to $s_{i}$ and take $1$ away from $s_{j+1}$ (hence the need for $i \neq j+1$), and the resulting die will beat $B_{n}$.
\begin{example}
\label{ex3}
Suppose player $A$ chooses die $X_{4}$ from our previous examples, $X_{4} = \big\{[1,1,4,4]\big\}$. We have $\gamma_{1} = 2$, $\gamma_{2} = \gamma_{3} = 0$, and $\gamma_{4} = 2$, and so $\xi_{1} = 2$, $\xi_{2} = 0$, $\xi_{3} = 2$, and $\xi_{4} = 2$. Evidently $\xi_{1} > \xi_{2}$ and $1 = i \neq j+1 = 2+1 = 3$. Hence, we add $1$ to $s_{1}$ and subtract $1$ from $s_{3}$, yielding the die $Y_{4} = \big\{[2,2,2,4]\big\}$, which beats $X_{4}$. Indeed, should player $B$ choose $Y_{4}$ she would achieve a payoff of $9/16 > 1/2$.
\end{example}

Evidently, if $\xi_{a} \neq \xi_{b}$ for some $a$, $b$, then there must be some $i$, $j$ with $\xi_{i} > \xi_{j}$. Thus, we establish the following lemma:

\begin{lemma}
\label{lem4}
If $n \geq 4$, then for any non-standard $n-$sided die $\exists$ a pair $a,b \in \big\{1,2,...,n\big\}$, for which $\xi_{a} \neq \xi_{b}$.
\end{lemma}

\begin{proof}
Evidently $\xi_{a} = \xi_{b} \hspace{.5cm} \forall a, b \in \big\{1,2,\dots,n\big\}$ if and only if

\[
\gamma_{1} + \gamma_{2} = \gamma_{2} + \gamma_{3} = \gamma_{3} + \gamma_{4} = \cdots = \gamma_{n-1} + \gamma_{n}\]
which holds if and only if

\begin{equation}
\label{eq:34}
\begin{split}
\gamma_{1} &= \gamma_{3} = \cdots = \gamma_{k} \hspace{1cm} \forall \hspace{.2cm} \text{odd integers} \hspace{.1cm} k \in \big\{1,2,\dots,n\big\}\\
\gamma_{2} &= \gamma_{4} = \cdots = \gamma_{j} \hspace{1cm} \forall \hspace{.2cm} \text{even integers} \hspace{.1cm} j \in \big\{1,2,\dots,n\big\}\\
\end{split}
\end{equation}
We also have the following two relationships:

\begin{equation}
\label{eq:35}\begin{split}
\sum_{k \hspace{.2cm} \text{odd}}^{n}\gamma_{k} + \sum_{j \hspace{.2cm} \text{even}}^{n}\gamma_{j} &= n
\end{split}
\end{equation}
and
\begin{equation}
\label{eq:57}\begin{split}
\sum_{k \hspace{.2cm} \text{odd}}^{n}k \gamma_{k} + \sum_{j \hspace{.2cm} \text{even}}^{n}j \gamma_{j} &= \frac{n(n+1)}{2}
\end{split}
\end{equation}
We can combine equations \ref{eq:34} and \ref{eq:35} to obtain

\begin{equation}
\label{eq:36}\begin{split}
\frac{n+1}{2}\gamma_{1} + \frac{n-1}{2}\gamma_{2} &= n\\
\end{split}
\end{equation}
for odd $n$. For even $n$, equations \ref{eq:34} and \ref{eq:35} yield \[\frac{n}{2}\gamma_{1} + \frac{n}{2}\gamma_{2} = n\]
or
\begin{equation}
\label{eq:37}\begin{split}
\gamma_{1} + \gamma_{2} &=2
\end{split}
\end{equation}
Furthermore, from equations \ref{eq:34} and \ref{eq:57}, we have \[\frac{n^{2}}{4}\gamma_{1} + \frac{n(n+2)}{4}\gamma_{2} = \frac{n(n+1)}{2}\]
or
\begin{equation}
\label{eq:137}\begin{split}
\frac{n}{2}\gamma_{1} + \frac{n+2}{2}\gamma_{2} &=n+1\\
\end{split}
\end{equation}
for even $n$. Now observe, we cannot have both $\gamma_{1} \geq 1$ and $\gamma_{2} \geq 1$ since if one were equal to $1$ and the other were greater than $1$, this would violate \ref{eq:35}; and if they were both equal to $1$, then $B_{n} = S_{n}$, a contradiction. Thus, either $\gamma_{1}$ or $\gamma_{2}$ must be equal to $0$.

Suppose $n$ is odd and that $\gamma_{1} = 0$. From equation \ref{eq:36} we have $(n-1)\gamma_{2} = 2n$, which does not have a solution in integers $n$, $\gamma_{2}$ for $n > 3$. Next, suppose $n$ is odd and that $\gamma_{2} = 0$. From equation \ref{eq:36} we have $(n+1)\gamma_{1} = 2n$, which does not have a solution in integers $n$, $\gamma_{1}$ for $n > 1$. Thus, we conclude that $n$ cannot be odd.

Suppose $n$ is even and that $\gamma_{1} = 0$. From equations \ref{eq:37} and \ref{eq:137} we must have $\gamma_{2} = 2$ and $n+2 = n+1$, which is obviously a contradiction. Finally, suppose $n$ is even and that $\gamma_{2} = 0$. From equations \ref{eq:37} and \ref{eq:137} we must have $\gamma_{1} = 2$ and $n = n+1$, which is also a contradiction. Thus we have proved Lemma \ref{lem4}.
\end{proof}

To wrap up the proof of Proposition \ref{prop2} we need to verify that we cannot have the situation where the only pair $\xi_{i},\xi_{j}$ that satisfies $\xi_{i} > \xi_{j}$ occurs when $i = j+1$. To that end, suppose $\xi_{i} > \xi_{j}$ for $i = j+1$. First, let $j \neq 1$. Then, if $\xi_{j-1} \leq \xi_{j}$, relabel $j-1$ as $j'$, which yields $\xi_{i} > \xi_{j'}$ for $i \neq j'+1$. On the other hand, if $\xi_{j-1} > \xi_{j}$, relabel $j-1$ as $i'$, implying $\xi_{i'} > \xi_{j}$ for $i' \neq j+1$. Next, let $j = 1$. If $\xi_{i+1} \geq \xi_{i}$, relabel $i+1$ as $i'$, which yields $\xi_{i'} > \xi_{j}$ for $i' \neq j+1$. If, instead, $\xi_{i+1} < \xi_{i}$, relabel $i+1$ as $j'$, and thus we have $\xi_{i} > \xi_{j'}$ for $i \neq j'+1$.
\end{proof}

We may also write the following Corollary, which we have proved along the way.

\begin{corollary}
\label{corr2}
Let $n \geq 4$. Then, for any die $B_{n} \neq S_{n}$, $\exists$ a die $G_{n}$ that is the image of a one-step $\Phi_{i,j}$ on the standard die $S_{n}$ that beats $B_{n}$.
\end{corollary}

Note that given some die $B_{n} \neq S_{n}$, the algorithm developed in our proof yields \textit{every} winning die (i.e. a die that beats $B_{n}$) that is a one-step away from $S_{n}$. Moreover, it is easy to see how by ``flipping" the algorithm we could also obtain the set of losing dice that are a one-step away from $S_{n}$. Furthermore, the algorithm also enables us to find the ``best" (and ``worst") die/dice a one-step away from $S_{n}$ to play versus $B_{n}$. The die or dice with the greatest net gain from the one-step would have the highest likelihood of beating $S_{n}$. In the same manner, the die or dice with the greatest (in absolute value) net loss from the one-step would have the lowest likelihood of beating $S_{n}$.


\end{document}